\documentclass[11pt,reqno]{amsart}
\usepackage{cmap}
\usepackage[T1]{fontenc}
\usepackage[utf8]{inputenc}
\usepackage[english]{babel}
\usepackage{amsfonts,amssymb,amsthm,amsmath}
\usepackage{url}
\usepackage{enumitem}

\usepackage{secdot}

\usepackage{graphicx}
\usepackage{epstopdf}
\usepackage{a4wide}
\usepackage{xcolor}
\usepackage[colorlinks=true,linktocpage,pdfpagelabels,
bookmarksnumbered,bookmarksopen]{hyperref}
\definecolor{ForestGreen}{rgb}{0.1,0.6,0.05}
\definecolor{EgyptBlue}{rgb}{0.063,0.1,0.6}
\hypersetup{
	colorlinks=true,
	linkcolor=EgyptBlue,         
	citecolor=ForestGreen,
	urlcolor=olive
}

\usepackage[hyperpageref]{backref}

\newtheorem{theorem}{Theorem}
\newtheorem{proposition}[theorem]{Proposition}
\newtheorem{lemma}[theorem]{Lemma}
\newtheorem{corollary}[theorem]{Corollary}

\theoremstyle{definition}

\newtheorem{remark}[theorem]{Remark}

\let\OLDthebibliography\thebibliography
\renewcommand\thebibliography[1]{
	\OLDthebibliography{#1}
	\setlength{\parskip}{1pt}
	\setlength{\itemsep}{1pt plus 0.3ex}
}

\numberwithin{equation}{section}
\numberwithin{theorem}{section}
\numberwithin{equation}{section}
\numberwithin{theorem}{section}

\title[On asymptotic behaviour of Dirichlet inverse]{On asymptotic behaviour of Dirichlet inverse}

\author[F.~Baustian]{Falko Baustian}
\author[V.~Bobkov]{Vladimir Bobkov}

\address[F.~Baustian]{\newline\indent
	Department of Mathematics, University of Rostock,
	\newline\indent 
	Ulmenstra{\ss}e 69, 18057 Rostock, Germany
}
\email{falko.baustian@uni-rostock.de}

\address[V.~Bobkov]{\newline\indent
	Department of Mathematics and NTIS, Faculty of Applied Sciences,
	\newline\indent 
	University of West Bohemia, Univerzitn\'i 8, 301 00 Plze\v{n}, Czech Republic\\
	\newline\indent 
	Institute of Mathematics, Ufa Federal Research Centre, RAS,
	\newline\indent 
	Chernyshevsky str. 112, 450008 Ufa, Russia
}
\email{bobkov@kma.zcu.cz}

\date{}


\thanks{
	V.~Bobkov was supported by the grant 18-03253S of the Grant Agency of the Czech Republic and by the project LO1506 of the Czech Ministry of Education, Youth and Sports.
}

\begin{document}
\begin{abstract} 
	Let $f(n)$ be an arithmetic function with $f(1)\neq0$ and let $f^{-1}(n)$ be its reciprocal with respect to the Dirichlet convolution.
	We study the asymptotic behaviour of $|f^{-1}(n)|$ with regard to the asymptotic behaviour of $|f(n)|$ assuming that the latter one grows or decays with at most polynomial or exponential speed. 
	As a by-product, we obtain simple but constructive upper bounds for the number of ordered factorizations of $n$ into $k$ factors. 
\end{abstract} 
\maketitle 

\section{Introduction}\label{sec:intro}

Let $f: \mathbb{N} \mapsto \mathbb{R}$ be an arithmetic function. 
The set of those $f(n)$ with $f(1) \neq 0$ endowed with the Dirichlet convolution defined as
$$
(f \ast g)(n) = \sum_{d|n} f\left(\frac{n}{d}\right) g(d), \quad n \in \mathbb{N},
$$
forms an abelian group. The identity element $\varepsilon(n)$ is given by $\varepsilon(1)=1$, $\varepsilon(n)=0$ for all $n \geq 2$, and we denote by $f^{-1}(n)$ the corresponding inverse of $f(n)$, i.e.,
\begin{equation}\label{eq:ffe}
(f \ast f^{-1})(n) = (f^{-1} \ast f)(n) = \varepsilon(n), \quad n \in \mathbb{N}.
\end{equation}
We call $f^{-1}(n)$ \textit{the Dirichlet inverse} of $f(n)$ and note that $f^{-1}(n)$ can be determined recursively via \eqref{eq:ffe} as
\begin{equation}\label{eq:DI}
f^{-1}(1) = \frac{1}{f(1)}
\quad\text{and}\quad
f^{-1}(n) = - \frac{1}{f(1)}\sum_{\substack{d|n \\ d<n}} f\left(\frac{n}{d}\right) f^{-1}(d), \quad n \geq 2.
\end{equation}
Alternatively, $f^{-1}(n)$ can be found in the following nonrecurrent way:
\begin{equation}\label{eq:Dirinv1}
f^{-1}(n) = \sum_{k=1}^{\Omega(n)} \frac{(-1)^k}{f(1)^{k+1}} \sum_{\substack{d_1\cdots d_k=n \\ d_1,\cdots, d_k \geq 2}} f(d_1) \cdots f(d_k),
\quad n \geq 2,
\end{equation}
where $\Omega(n)$ is the number of prime factors of $n$ counted with multiplicities. The formula \eqref{eq:Dirinv1} can be obtained from \cite[Theorem 2.2]{hauk} using the evident identity
\begin{equation}\label{eq:rescaling}
(a f)^{-1}(n) = \frac{1}{a} f^{-1}(n), \quad n \geq 1, \quad a \in \mathbb{R} \setminus \{0\}.
\end{equation}
As a consequence of \eqref{eq:rescaling}, \textit{we will always assume that} $f(1)=1$, unless otherwise stated. 
In particular, as was noticed by Hille \cite{hille}, taking $f(n)=-1$ for all $n \geq 2$, one gets $f^{-1}(n)=H(n)$, where 
\begin{equation}\label{eq:HnHk}
H(n) 
=
\sum_{k=1}^{\Omega(n)} H_k(n)
= 
\sum_{k=1}^{\Omega(n)} \sum_{\substack{d_1\cdots d_k=n \\ d_1,\cdots, d_k \geq 2}} 1, \quad n \geq 2,
\end{equation}
is \textit{the number of ordered factorizations of $n$} and $H_k(n)$ is \textit{the number of ordered factorizations of $n$ into $k$ factors} where each factor is greater than or equal to $2$.

\medskip
In the analysis of various problems there appears a necessity to control the growth or decay rate of both $f(n)$ and $f^{-1}(n)$, simultaneously. For instance, Segal showed in \cite{segal1} that if $f(n) = O(1)$ and $f^{-1}(n) = O(1)$ as $n \to \infty$, along with other assumptions, then $\sum_{n \leq x} l(n) \sim x$, where $l(n)$ represents the coefficients of the Dirichlet series
$$
-\frac{\mathcal{D}'(s)}{\mathcal{D}(s)} = \sum_{n=1}^{\infty} \frac{l(n)}{n^s}
\quad \text{with} \quad 
\mathcal{D}(s) = \sum_{n=1}^{\infty} \frac{f(n)}{n^s}.
$$
This result can be seen as an analogue of the prime number theorem. 
On the other hand, Segal proposed in \cite{segal2} the following generalization of Ingham's summation method \cite{ingham,wintner}: a series $\sum_{n=1}^{\infty}a_n$ is said to be $(D, f(n))$-summable to $A \in \mathbb{R}$ whenever
$$
\lim\limits_{x \to \infty} D(x) = A,
\quad \text{where}~ 
D(x) = \frac{1}{x} \sum_{n \leq x} n \sum_{d | n} a_d f\left(\frac{n}{d}\right).
$$
Properties of the $(D, f(n))$-summation method crucially depend on the summability of $f(n)$ and $f^{-1}(n)$, and, as a consequence, on their asymptotic behaviour, see, e.g., \cite{jukes,segal2}. 

Assume now that $f(n)$ is given by the Fourier coefficients of a function $F \in L^2(0,1)$ which is extended to the whole $\mathbb{R}$ antiperiodically with period $1$. 
In analogy with properties of the standard trigonometric system $\{\sin(n \pi x)\}$ it is natural to ask which assumptions one should impose on $F$ in order to guarantee that the system
\begin{equation}\label{eq:F}
F(x),~ F(2x),~ F(3x),~ \dots,
\end{equation} 
forms a basis in $L^2(0,1)$ or at least is complete in the same space.
Here, by completeness of \eqref{eq:F} we mean that any function from $L^2(0,1)$ can be approximated in the $L^2$-norm with an arbitrary precision by finite linear combinations of functions \eqref{eq:F}.
This problem has been intensively studied, see, e.g., historical remarks in \cite{lind2,wintner2}. 
In particular, the following result was obtained by Hedenmalm et al.\ in \cite{lind}.
\begin{theorem}[\protect{\cite[Theorem 5.7 (a)]{lind}}]\label{thm:lind}
	Let $F \in L^2(0,1)$ be such that 
	\begin{equation}\label{eq:fn-sum}
	\sum_{n=1}^{\infty} \left|f(n)\right|^2 \tau(n) < \infty 
	\quad
	\text{and}
	\quad
	\sum_{n=1}^{\infty} \left|f^{-1}(n)\right|^2 \tau(n) < \infty,
	\end{equation}
	where $\tau(n)$ is the number of divisors of $n$.
	Then the system $\{F(nx)\}$ is complete in $L^2(0,1)$.
\end{theorem}
It is well-known that $\tau(n) = o(n^\delta)$ for any $\delta>0$, see, e.g., \cite[p.\ 296]{apostol}. Thus, the assumptions \eqref{eq:fn-sum} can be easily verified provided
\begin{equation*}\label{eq:fn<infty}
\left|f(n)\right| \leq C_1 n^{-\frac{1}{2}-\varepsilon}
\quad
\text{and}
\quad
\left|f^{-1}(n)\right| \leq C_2 n^{-\frac{1}{2}-\eta},
\quad n \geq 1,
\end{equation*}
for some $C_1, C_2,\varepsilon,\eta>0$.
See also \cite[pp.\ 764-765]{wintner2} and \cite{sowa} for similar assumptions guaranteeing that the system \eqref{eq:F} forms a basis in $L^2(0,1)$. We recognize that the asymptotics of $f(n)$ and $f^{-1}(n)$ can be used in the study of such kind of problems, as well.

\medskip
Although the asymptotic behaviour of $f(n)$ can be considered as given or relatively easy to obtain, the asymptotic behaviour of $f^{-1}(n)$ is, in general, a hard issue and it can be drastically different from those of $f(n)$. 
As an example, assume that $f(2)=-1$ and $f(n)=0$ for all $n \geq 3$. 
Then we easily see from \eqref{eq:Dirinv1} that $f^{-1}(2^k)=1$ for all $k \geq 1$, and $f^{-1}(n)=0$ for any $n$ with a prime factor different from $2$. That is, $|f^{-1}(n)|$ does not have to converge to $0$ as $n \to \infty$ even if $|f(n)|$ decays arbitrarily fast. Clearly, this is due to the definition of $f^{-1}(n)$ from which we see that the asymptotic behaviour of $f^{-1}(n)$ depends on values of $f(n)$ for all $n$ rather than only for sufficiently large $n$.   
 
Nevertheless, under additional requirements, the asymptotic behaviour of $f^{-1}(n)$ can be explicitly controlled by or compared with those of $f(n)$. 
Perhaps, the simplest case of this type occurs if $f(n)$ is assumed to be \textit{totally (completely) multiplicative}, i.e.,
\begin{equation}\label{eq:total}
f(m) f(n) = f(mn)
\quad \text{for all }
m,n \in \mathbb{N}.
\end{equation}
Then the Dirichlet inverse $f^{-1}(n)$ has the following explicit form:
$$
f^{-1}(n) = \mu(n) f(n), \quad n \geq 1,
$$
where $\mu(n)$ is the M\"obius function defined by $\mu(1)=1$ and
$$
\mu(n) =
\left\{
\begin{aligned}
&0 &&\text{if $n$ has a squared prime factor},\\
&(-1)^r &&\text{if $n$ is the product of $r$ distinct primes},
\end{aligned}
\right.
\qquad
n \geq 2,
$$
see, e.g., \cite[Theorem 2.17]{apostol}.
Therefore, $|\mu(n)| \leq 1$ and hence 
$$
|f^{-1}(n)| \leq |f(n)|
\quad 
\text{for all}~ n \geq 1.
$$ 
In other words, $|f^{-1}(n)|$ cannot grow faster or decay slower than $|f(n)|$. 

\medskip
However, in general position, $f(n)$ is not totally multiplicative, and to the best of our knowledge there are not many results connecting the asymptotic behaviour of $f^{-1}(n)$ with those of $f(n)$ without assuming \eqref{eq:total}, see, e.g., \cite{hauk2,kacper,sowa} for some particular classes of $f(n)$. 
The aim of the present article is thus to investigate the assumptions on $f(n)$ under which the explicit control of the behaviour of $f^{-1}(n)$ is possible. We will concentrate on the cases where $f(n)$ has at most polynomial or exponential speed as $n \to \infty$, that is, 
$$
\text{either}
\quad  
|f(n)| \leq C n^\gamma 
\quad
\text{or}
\quad
|f(n)| \leq A c^n, 
\quad n \geq 2,
$$
for some $C>0$, $\gamma \in \mathbb{R}$, and $A, c>0$.

\medskip
This article is organized as follows.
In Section \ref{sec:Hkn}, we obtain some estimates on $H(n)$, $H_k(n)$, and their generalizations which will be used in the sequel but also have an independent interest. 
In Section \ref{sec:decay}, we present and prove our main results concerning the asymptotic behaviour of $f^{-1}(n)$. 
We consider several weakenings of the total multiplicativity assumption \eqref{eq:total} in Section \ref{sec:subsupermult}, the general case is studied in Section \ref{sec:general}, and we conclude the article with some miscellaneous cases in Section \ref{sec:misc}.

\section{Number of ordered factorizations}\label{sec:Hkn}
In this section, we give some upper bounds on the functions $H(n)$ and $H_k(n)$ defined by \eqref{eq:HnHk} and on their generalizations. 
Let $\mathcal{P}$ be a subset of $\mathbb{N}_2 = \{2,3,\dots\} \subset \mathbb{N}$. 
Denote by $H(n,\mathcal{P})$ the number of ordered factorizations of $n$ where each factor belongs to $\mathcal{P}$, and by $H_k(n,\mathcal{P})$ the corresponding number of ordered factorizations of $n$ into $k$ factors, that is,
\begin{equation*}\label{eq:hnp}
H(n,\mathcal{P}) 
=
\sum_{k=1}^{\Omega(n)} H_k(n,\mathcal{P})
= 
\sum_{k=1}^{\Omega(n)} \sum_{\substack{d_1\cdots d_k=n \\ d_1,\cdots, d_k \in \mathcal{P}}} 1, 
\quad n\geq 2.
\end{equation*}
In particular, if $\mathcal{P} = \mathbb{N}_2$, then $H(n,\mathcal{P}) = H(n)$ and $H_k(n,\mathcal{P}) = H_k(n)$. 
The functions $H(n,\mathcal{P})$ and $H_k(n,\mathcal{P})$ have been intensively studied starting from the work of Kalm\'ar \cite{kalmar}, see, e.g., \cite{chor,cop,hille,hwangjan} and overviews \cite{knopf,spri}.

We start with several standard observations.
Notice that 
\begin{equation}\label{eq:Hkrec1}
H_1(n, \mathcal{P}) 
=
\left\{ 
\begin{aligned}
&1 &\text{if}~ n \in \mathcal{P},\\
&0 &\text{if}~ n \not\in \mathcal{P}.
\end{aligned}
\right.
\end{equation}
Let us denote by $\zeta_\mathcal{P}(s)$ the Dirichlet series associated with $H_1(n, \mathcal{P})$, that is, 
$$
\zeta_\mathcal{P}(s) = \sum_{m=1}^{\infty} \frac{H_1(m, \mathcal{P})}{m^s} = \sum_{m \in \mathcal{P}} \frac{1}{m^s}.
$$
In particular, if $\mathcal{P} = \mathbb{N}_2$, then $\zeta_\mathcal{P}(s) = \zeta(s)-1$, where $\zeta(s)$ is the Riemann zeta function. 

Let $\sigma_\mathcal{P}$ be the abscissa of convergence of $\zeta_\mathcal{P}(s)$. 
If $\mathcal{P}$ has a finite cardinality, then $\sigma_\mathcal{P} = -\infty$, while if $\mathcal{P}$ is infinite, then $\sigma_\mathcal{P} \in [0,1]$. 
For our further purposes, we will be interested only in real $s \geq 0$.
Clearly, $\zeta_\mathcal{P}(s)$ decreases, $\zeta_\mathcal{P}(s) \to 0$ as $s \to \infty$, and there exists $s_0 \geq \max\{\sigma_\mathcal{P},0\}$ such that $\zeta_\mathcal{P}(s_0) \geq 1$.
Hereinafter, we will denote by $\rho(\mathcal{P})$ the unique real root of $\zeta_\mathcal{P}(s) = 1$, $s  \geq s_0$. 

The functions $H(n,\mathcal{P})$ and $H_k(n,\mathcal{P})$ can be determined recursively by
\begin{equation}\label{eq:Hrec}
H(1,\mathcal{P}) = 1 
\quad \text{and} \quad
H(n,\mathcal{P}) = \sum_{\substack{d|n\\ d \in \mathcal{P}}} H\left(\frac{n}{d},\mathcal{P}\right),
\quad n \geq 2,
\end{equation}
and
\begin{equation}\label{eq:Hkrec}
H_k(n,\mathcal{P}) = \sum_{\substack{d|n\\ d \in \mathcal{P}}} H_{k-1}\left(\frac{n}{d},\mathcal{P}\right),
\quad k \geq 2,
\end{equation}
where $H_1(n, \mathcal{P})$ is given by \eqref{eq:Hkrec1}.
We see that 
$$
H_k(n, \mathcal{P}) = H_{k-1}(n, \mathcal{P}) \ast H_1(n, \mathcal{P}) = H_{k-2}(n, \mathcal{P}) \ast H_1(n, \mathcal{P}) \ast H_1(n, \mathcal{P}) = \dots
$$
Thus, considering the Dirichlet series associated with $H_k(n, \mathcal{P})$, we obtain 
\begin{equation}\label{eq:DirichletH0}
\sum_{m=1}^{\infty} \frac{H_k(m,\mathcal{P})}{m^s} = \zeta_\mathcal{P}(s)^k, \quad k \geq 1.
\end{equation}

\begin{lemma}\label{lem:upperH0}
	Let $n \geq 1$. Then 
	\begin{equation*}\label{eq:upperH00}
	H(n, \mathcal{P}) \leq n^{\rho(\mathcal{P})}.
	\end{equation*}
\end{lemma}	
\begin{proof}
	We will argue in much the same way as in \cite[Section 2]{cop} and prove the result by induction. 
	The base of induction is trivial. 
	Take some $n \geq 2$ and suppose that 
	$$
	H(m, \mathcal{P}) \leq m^{\rho(\mathcal{P})}
	\quad \text{for all}~ m < n.
	$$
	Let us show that the inequality remains valid for $m=n$.
	Using \eqref{eq:Hrec}, we deduce that 
	\begin{align*}
	H(n, \mathcal{P}) 
	= 
	\sum_{\substack{d|n\\ d \in \mathcal{P}}} H\left(\frac{n}{d},\mathcal{P}\right)
	\leq
	\sum_{\substack{d|n\\ d \in \mathcal{P}}} \left(\frac{n}{d}\right)^{\rho(\mathcal{P})} \leq 
	n^{\rho(\mathcal{P})} \sum_{d \in \mathcal{P}} \frac{1}{d^{\rho(\mathcal{P})}}
	=
	n^{\rho(\mathcal{P})} \zeta_\mathcal{P}(\rho(\mathcal{P})) 
	= 
	n^{\rho(\mathcal{P})},
	\end{align*}
	which completes the proof.
\end{proof}

Now we obtain an upper bound for $H_k(n, \mathcal{P})$.
\begin{lemma}\label{lem:upperH}
	Let $s>\sigma_\mathcal{P}$, $n \geq 2$, and $k \geq 1$. Then
	\begin{equation}\label{eq:upperH0}
	H_k(n, \mathcal{P}) \leq \frac{\zeta_\mathcal{P}(s)^{k-1} n^s}{\varrho^s},
	\end{equation}
	where $\varrho$ is the minimal element of $\mathcal{P}$. 
\end{lemma}	
\begin{proof}
	Inequality \eqref{eq:upperH0} is trivial for $k=1$, see \eqref{eq:Hkrec1}.
	Assuming $k \geq 2$, we use \eqref{eq:Hkrec} and \eqref{eq:DirichletH0} to deduce that 
	\begin{align*}
	\frac{H_k(n, \mathcal{P})}{n^s} = \sum_{\substack{d|n\\ d \in \mathcal{P}}} \left(\frac{d}{n}\right)^s \frac{H_{k-1}\left(\frac{n}{d},\mathcal{P}\right)}{d^s}
	\leq
	\frac{1}{\varrho^s} \sum_{m=1}^{\infty}\frac{H_{k-1}(m, \mathcal{P})}{m^s}
	=
	\frac{\zeta_\mathcal{P}(s)^{k-1}}{\varrho^s},
	\quad n \geq 2,
	\end{align*}
	where we applied the inequality $d \geq \varrho$ for $d \in \mathcal{P}$.
\end{proof}
	
Let us provide two corollaries of Lemmas \ref{lem:upperH0} and \ref{lem:upperH} where we treat the cases $\mathcal{P} = \mathbb{N}_2$ and $\mathcal{P} = \mathbb{N}_3^{\text{odd}} = \{3,5,7,\dots\}$ which is the set of all odd natural numbers except $1$. The latter case occurs naturally in the basisness and completeness problems, see Section \ref{sec:fodd} for further discussion.
One can show (cf.\ \cite[Proposition 1]{chor}) that
\begin{equation}\label{eq:zetaodd}
\zeta_{\mathbb{N}_3^{\text{odd}}}(s) 
= 
\sum_{\substack{m \geq 3\\ m ~\text{odd}}} \frac{1}{m^s}
=
\prod_{\substack{p ~\text{prime}\\ p \geq 3}} \frac{1}{1-p^{-s}} - 1
=
\left(1-\frac{1}{2^s}\right)\zeta(s)-1,
\end{equation}
and find that $\rho(\mathbb{N}_3^{\text{odd}}) = 1.37779\dots$
\begin{corollary}\label{cor:lem:Hodd}
	Let $n \geq 2$. 
	Then $H(n) \leq n^\rho$ for all $n \geq 2$, where $\rho = \rho(\mathbb{N}_2) = 1.72865\dots$ 
	If, in addition, $n$ is odd, then $H(n) \leq n^\eta$, where $\eta = \rho(\mathbb{N}_3^{\text{odd}}) = 1.37779\dots$
\end{corollary}

\begin{remark}\label{rem:Hn}
	In fact, the inequalities in Corollary \ref{cor:lem:Hodd} are strict, see \cite[Theorem 5]{chor}. 
	Furthermore, the growth rate $n^\rho$ is optimal. For any $\varepsilon>0$ there exist infinitely many $n$ such that $H(n) > n^{\rho-\varepsilon}$, see \cite{hille} and also \cite[Section 3]{cop} for an explicit construction.
\end{remark}

\begin{corollary}\label{cor:lem:H}
	Let $s>1$, $n \geq 2$, and $k \geq 1$. Then
	\begin{equation}\label{eq:upperH}
	H_k(n) \leq \frac{\left(\zeta(s)-1\right)^{k-1} n^s}{2^s}.
	\end{equation}
	Moreover, if $n$ is odd, then
	\begin{equation*}\label{eq:upperHodd}
	H_k(n) \leq \frac{\left(\left(1-\frac{1}{2^s}\right)\zeta(s)-1\right)^{k-1} n^s}{3^s}.
	\end{equation*}
\end{corollary}

\section{Asymptotics of \texorpdfstring{$f^{-1}(n)$}{f-1(n)}}\label{sec:decay}
In this section, we state and prove our main results concerning the behaviour of $f^{-1}(n)$. 

\subsection{Weakening of the total multiplicativity}\label{sec:subsupermult}
First, we weaken the total multiplicativity assumption \eqref{eq:total} in the following two ways.
We say that the absolute value of $f(n)$ is \textit{supermultiplicative} if 
\begin{equation}\label{eq:submult}
|f(m)| |f(n)| \leq |f(mn)| 
\quad \text{for all}~ m,n \in \mathbb{N},
\end{equation}
and it is \textit{submultiplicative} if 
\begin{equation}\label{eq:submult1}
|f(m)| |f(n)| \geq |f(mn)| 
\quad \text{for all}~ m,n \in \mathbb{N}.
\end{equation}

Recalling that we always assume $f(1)=1$, we use \eqref{eq:Dirinv1} to derive that if $|f(n)|$ is supermultiplicative, then 
$$
|f^{-1}(n)| \leq H(n) |f(n)|, \quad n \geq 2,
$$
while if $|f(n)|$ is submultiplicative, then 
\begin{equation}\label{eq:submult2}
|f^{-1}(n)| \leq H(n) \prod_{j=1}^{\omega(n)} |f(p_j)|^{e_j}, \quad n \geq 2.
\end{equation}
Here and below, we write arbitrary $n \in \mathbb{N}$ as its prime decomposition 
$$
n=p_1^{e_1}\cdots p_{\omega(n)}^{e_{\omega(n)}},
$$
where $\omega(n)$ stands for the number of distinct prime factors of $n$.

In particular, if $f(n)$ has a polynomial behaviour, then, in view of Corollary \ref{cor:lem:Hodd},  \eqref{eq:submult2} can be estimated from above in the following more explicit way. 
\begin{proposition}\label{prop:submult}
	Let $|f(n)|$ be submultiplicative. Assume that there exist $C >0$ and $\gamma \in \mathbb{R}$ such that $|f(n)| \leq C n^\gamma$ for all $n\geq 2$. Then
	\begin{equation}\label{eq:submult3}
	|f^{-1}(n)| \leq H(n) C^{\Omega(n)} n^{\gamma} \leq C^{\Omega(n)} n^{\gamma+\rho},
	\quad
	n \geq 2,
	\end{equation}
	where $\rho = 1.72865\dots$ is the unique root of $\zeta(s)=2$.
\end{proposition}
We refer to Remark \ref{rem:optimal} below for a discussion of the optimality of \eqref{eq:submult3}.

\medskip
Let us consider another weakening of \eqref{eq:total}.
We call $f(n)$ \textit{multiplicative} if 
\begin{equation}\label{eq:mult}
f(m)f(n) = f(mn) 
\quad \text{for all coprime}~ m,n \in \mathbb{N}.
\end{equation}
Note that the Dirichlet inverse $f^{-1}(n)$ of a multiplicative $f(n)$ is also multiplicative (see, e.g., \cite[Theorem 2.16]{apostol}), that is, 
\begin{equation}\label{eq:f-1compl0}
f^{-1}(n)
=
f^{-1}\left(\prod_{j=1}^{\omega(n)}p_j^{e_j}\right)
=
\prod_{j=1}^{\omega(n)}f^{-1}(p_j^{e_j}).
\end{equation}
In its turn, the Dirichlet inverse for prime powers can be found recursively as follows (see \eqref{eq:DI}):
\begin{equation}\label{eq:DImult}
f^{-1}(p) = -f(p) \quad \text{and} \quad 
f^{-1}(p^k) = -\sum_{m=0}^{k-1} f(p^{k-m}) f^{-1}(p^m),
\quad k \geq 2.
\end{equation}
Let us also refer the reader to \cite[(2.5)-(2.7)]{hauk} for nonrecurent expressions of $f^{-1}(p^k)$.
Notice that all interim results for powers of primes that we present in this section are also valid for non-multiplicative $f(n)$.

We divide our results on the multiplicative case into two subsections according to polynomial and exponential behaviour of $f(n)$.
\subsubsection{Polynomial behaviour}
Along this subsection, we will assume that $|f(n)| \leq C n^\gamma$ for some $C>0$, $\gamma \in \mathbb{R}$, and all $n \geq 2$. 
We start from the following general result.
\begin{proposition}\label{prop:mult}
	Let $f(n)$ be multiplicative. Assume that there exist $C >0$ and $\gamma \in \mathbb{R}$ such that $|f(n)| \leq C n^\gamma$ for all $n\geq 2$. Then
	\begin{equation*}
	|f^{-1}(n)| \leq \left(\frac{C}{C+1}\right)^{\omega(n)}(C+1)^{\Omega(n)} n^{\gamma},
	\quad
	n \geq 2.
	\end{equation*}
\end{proposition}
\begin{proof}
	Assume first the case $n=p^k$ for prime $p\geq2$ and natural $k \geq 1$.
	Let us argue by induction with respect to $k$. To show the base of induction, we recall that $f^{-1}(p) = -f(p)$, which implies $|f^{-1}(p)| = |f(p)| \leq C p^\gamma$, see \eqref{eq:DImult}.
	As the hypothesis of induction, we assume that $|f^{-1}(p^m)| \leq C(C+1)^{m-1} p^{m\gamma}$ for all $m \leq k-1$.
	We perform the inductive step using \eqref{eq:DImult}:
	\begin{align*}
	|f^{-1}(p^k)|
	&\leq
	|f(p^k)| + \sum_{m=1}^{k-1}|f(p^{k-m})||f^{-1}(p^m)|
	\\
	&\leq
	Cp^{k\gamma} + C^2p^{k\gamma}\sum_{m=0}^{k-2}(C+1)^m
	=
	C(C+1)^{k-1}p^{k\gamma}.
	\end{align*}
	Hence, we have $|f^{-1}(p^k)| \leq C(C+1)^{k-1} p^{k\gamma}$ for all powers of primes. 
	Therefore, since $f^{-1}(n)$ is multiplicative, we derive from \eqref{eq:f-1compl0} that 
	$$
	|f^{-1}(n)|
	=
	\prod_{j=1}^{\omega(n)} |f^{-1}(p_j^{e_j})|
	\leq
	\prod_{j=1}^{\omega(n)} C(C+1)^{e_j-1}p_j^{e_j\gamma}
	=
	\left(\frac{C}{C+1}\right)^{\omega(n)}(C+1)^{\Omega(n)} n^{\gamma}
	$$
	for arbitrary natural $n \geq 2$.
\end{proof}
Since $\Omega(n)$ possesses the upper bound $\frac{\ln n}{\ln 2}$, we can simplify the statement of Proposition \ref{prop:mult}. The resulting corollary shows that the asymptotic behavior of the Dirichlet inverse $f^{-1}(n)$ depends directly\label{key} on the value of the constant $C>0$.
\begin{corollary}
	Let $f(n)$ be multiplicative. Assume that there exist $C >0$ and $\gamma \in \mathbb{R}$ such that $|f(n)| \leq C n^\gamma$ for all $n\geq 2$. Then $|f^{-1}(n)| \leq n^{\gamma+\frac{\ln(1+C)}{\ln 2}}$ for all $n \geq 2$.
\end{corollary}

\begin{remark}
	The upper bound obtained by Proposition \ref{prop:mult} is optimal. 
	Indeed, fix any $\gamma \in \mathbb{R}$ and $C>0$, and define $f(n)$ by $f(2^k) = -C 2^{k \gamma}$, $k \geq 1$, and $f(n)=0$ if $n > 2$ is not a power of $2$. 
	We see that such $f(n)$ is multiplicative.
	Using \eqref{eq:DImult}, we obtain by induction that
	$$
	f^{-1}(2^k)
	= 
	C \sum_{m=0}^{k-1} 2^{(k-m)\gamma} f^{-1}(p^m)
	=
	C 2^{k\gamma} + C^2 \sum_{m=1}^{k-1}(C+1)^{m-1} 2^{k\gamma}
	=
	C(C+1)^{k-1}2^{k\gamma}
	$$
	for all $k \geq 1$, 
	and $f^{-1}(n)=0$ for all other natural numbers $n>2$, which means the claimed optimality. 
\end{remark}

\medskip
By imposing additional assumptions on $f(n)$, one can improve the upper bound in Proposition \ref{prop:mult}.
For instance, assume that $f(n)$ is multiplicative and $f(p^k)=0$ for all primes $p \geq 2$ and naturals $k\geq 2$. 
We obtain recursively from \eqref{eq:DImult} that $f^{-1}(p^k)=(-1)^kf(p)^k$ for $k \geq 1$.
Therefore, for arbitrary natural $n \geq 2$, \eqref{eq:f-1compl0} implies
\begin{equation}\label{eq:f-1compl}
f^{-1}(n)
=
\prod_{j=1}^{\omega(n)}f^{-1}(p_j^{e_j})
=
(-1)^{\Omega(n)}\prod_{j=1}^{\omega(n)}f(p_j)^{e_j},
\end{equation}
which shows that $f^{-1}(n)$ is totally multiplicative, cf.\ \cite[Exercise 26, p.\ 49]{apostol}.
Let us remark that this is a complementation to the totally multiplicative case discussed in Section \ref{sec:intro}.
As a consequence of \eqref{eq:f-1compl}, we get the following result.
\begin{proposition}\label{prop:mult2}
	Let $f(n)$ be multiplicative and $f(p^k)=0$ for all prime powers $p^k$ with $k\geq 2$.
	Assume that there exist $C >0$ and $\gamma \in \mathbb{R}$ such that $|f(n)| \leq C n^\gamma$ for all $n\geq 2$. Then
	$$
	|f^{-1}(n)|
	\leq
	C^{\Omega(n)}n^{\gamma},
	\quad
	n \geq 2.
	$$
\end{proposition}

\subsubsection{Exponential behaviour}
Along this subsection, we will assume that $|f(n)| \leq A c^n$ for some $A, c>0$ and all $n \geq 2$. 
Let us denote by $\mathfrak{P}(m)$ the set of all partitions of $m \in \mathbb{N}$ and assume, without loss of generality, that each entry of $\mathfrak{P}(m)$ is arranged in the decreasing order, e.g., 
$$
\mathfrak{P}(5) = \left\{(5),(4,1),(3,2),(3,1,1),(2,2,1),(2,1,1,1),(1,1,1,1,1)\right\}.
$$
Using the recursive formula \eqref{eq:DImult}, one can derive the following upper bound for the Dirichlet inverse for prime powers:
\begin{equation}\label{eq:upper:f-1:expt}
|f^{-1}(p^k)| 
\leq 
\sum_{(\phi_1,\phi_2,\dots,\phi_l) \in \mathfrak{P}(k)} 
\binom{l}{l_1,l_2,\dots,l_m}
A^l
\prod_{j=1}^l c^{p^{\phi_j}},
\end{equation}
where $l_1, l_2, \dots, l_m \geq 1$ are such that $l_1+\dots+l_{m} = l$ and
$$
\phi_1 = \dots = \phi_{l_1} > \phi_{l_1+1} = \dots = \phi_{l_1+l_2} > \dots > \phi_{l_1+\dots+l_{m-1}+1} = \dots = \phi_{l_1+\dots+l_{m}} = \phi_l,
$$
and $\binom{l}{l_1,l_2,\dots,l_m}$ is the multinomial coefficient. 
Let us remark that the inequality \eqref{eq:upper:f-1:expt} turns to equality for any function $f(n)$ satisfying $f(p^k)=-A c^{p^k}$ for prime powers $p^k$, $k \geq 1$. 
That is, \eqref{eq:upper:f-1:expt} gives a sharp upper bound for the Dirichlet inverse for prime powers.

If we assume now that $f(n)$ is multiplicative, then \eqref{eq:f-1compl0} and \eqref{eq:upper:f-1:expt} yield
\begin{equation}\label{eq:upper:f-1:expt1}
|f^{-1}(n)| 
\leq 
\prod_{i=1}^{\omega(n)}\sum_{(\phi_1,\phi_2,\dots,\phi_l) \in \mathfrak{P}(e_i)} 
\binom{l}{l_1,l_2,\dots,l_m}
A^l
\prod_{j=1}^l c^{p_i^{\phi_j}}, 
\quad n \geq 2.
\end{equation}
Although this upper bound is optimal and explicit, its application to particular choices of $f(n)$ can be complicated. 
Let us provide a simpler upper bound for \eqref{eq:upper:f-1:expt1} in the case $c \in (0,1)$.
\begin{proposition}\label{prop:mult:exp}
	Let $f(n)$ be multiplicative. Assume that there exist $A>0$ and $c \in (0,1)$ such that $|f(n)| \leq A c^n$ for all $n\geq 2$. Then
	\begin{equation}\label{eq:expf1}
	|f^{-1}(n)|
	\leq
	\left(\frac{A}{A+1}\right)^{\omega(n)}(A+1)^{\Omega(n)} n^\frac{3 \ln c}{\ln 3}, \quad n \geq 2.
	\end{equation}
\end{proposition}
\begin{proof}
	Taking any prime power $p^k$, a partition $(\phi_1,\phi_2,\dots,\phi_l) \in \mathfrak{P}(k)$ as above, and applying \cite[Lemma 2.2]{jakim}, we get
	$$
	\prod_{j=1}^l c^{p_i^{\phi_j}} 
	=
	c^{\sum_{j=1}^l p^{\phi_j}}
	\leq c^\frac{3 k \ln p}{\ln 3}.
	$$
	On the other hand, we have
	\begin{equation}\label{eq:upperest1}
	\sum_{(\phi_1,\phi_2,\dots,\phi_l) \in \mathfrak{P}(k)} 
	\binom{l}{l_1,l_2,\dots,l_m} A^l
	= 
	\sum_{l=1}^k 
	\binom{k-1}{l-1} A^l
	=
	A (A+1)^{k-1} 
	\end{equation}
	since the sums in \eqref{eq:upperest1} (considered without $A^l$) correspond to the number of compositions of $k$ into exactly $l$ parts. 
	Therefore, \eqref{eq:upper:f-1:expt} implies that $|f^{-1}(p^k)| \leq A (A+1)^{k-1} c^\frac{3 k \ln p}{\ln 3}$ for all prime powers.
	Finally, using the multiplicativity of $f^{-1}(n)$, we derive the inequality \eqref{eq:expf1} similarly as in the proof of Proposition \ref{prop:mult}. 
\end{proof}
\begin{corollary}\label{cor:mult:exp}
	Let $f(n)$ be multiplicative. 
	Assume that there exist $A>0$ and $c \in (0,1)$ such that $|f(n)| \leq A c^n$ for all $n\geq 2$.  
	Then $|f^{-1}(n)| \leq n^{\frac{3 \ln c}{\ln 3}+\frac{\ln(1+A)}{\ln 2}}$ for all $n \geq 2$.
\end{corollary}

\begin{remark}
	Even if $|f(n)|$ decays exponentially as $n \to \infty$, the same exponential decay of $|f^{-1}(n)|$ cannot be guaranteed, as was already discussed in Section \ref{sec:intro}. 
	In certain cases, a polynomial upper bound for $|f^{-1}(n)|$ gives the best achievable asymptotic behaviour. 
	To show this, we fix any $A>0$ and $c \in (0,1)$ and consider the multiplicative $f(n)$ defined as $f(2)=-A c^2$ and $f(n)=0$ for all $n \geq 3$.
	Then we see that $f^{-1}(n)=0$ for $n \neq 2^k$, and 
	$$
	f^{-1}(2^k) = A^{k} c^{2k} = 2^{\frac{k \ln A}{\ln 2} + \frac{2k \ln c}{\ln 2}}
	$$
	which reads as $f^{-1}(n) = n^\frac{\ln (Ac^2)}{\ln 2}$ for $n = 2^k$. That is, we have an explicit polynomial decay.
\end{remark}

\begin{remark}
	In the case $c>1$, the growth of $|f^{-1}(n)|$ cannot be expected to be slower than the exponential growth. 
	Let us consider the multiplicative function $f(n)$ defined by $f(2^k)=-A c^{2^k}$ for $k \geq 1$ and $f(n)=0$ for all other $n > 2$. Clearly, we have $f^{-1}(n)=0$ if $n\neq 2^k$, $f^{-1}(2)=Ac^2$, $f^{-1}(4)=Ac^4 + A^2c^4$, 
	and
	$$
	f^{-1}(2^k)=Ac^{2^k} + \mbox{lower order terms}, \quad k\geq3.
	$$
	An exponential upper bound for $|f^{-1}(n)|$ will be given for general $f(n)$ in Section \ref{sec:general:exp} below.
\end{remark}

\subsection{General case}\label{sec:general}
In this section, we consider the asymptotic behaviour of $f^{-1}(n)$ regardless the assumptions \eqref{eq:total}, \eqref{eq:submult}, \eqref{eq:submult1}, and \eqref{eq:mult}. 
Again, we divide our results into two subsections according to polynomial and exponential behaviour of $f(n)$, respectively.

\subsubsection{Polynomial behaviour}
Along this subsection, we will assume that $|f(n)| \leq C n^\gamma$ for some $C>0$, $\gamma \in \mathbb{R}$, and all $n \geq 2$. Using this upper bound and Corollary \ref{cor:lem:H}, we readily deduce from \eqref{eq:Dirinv1} that
\begin{align}\label{eq:f-1up}
\begin{split}
|f^{-1}(n)| 
&\leq 
n^\gamma \sum_{k=1}^{\Omega(n)} C^k H_k(n) 
\leq
\frac{C n^{\gamma+\varsigma}}{2^\varsigma} \sum_{k=1}^{\Omega(n)} \left(C\left(\zeta(\varsigma)-1\right)\right)^{k-1}
\\
&=
\frac{\Omega(n) \cdot C n^{\gamma+\varsigma}}{2^\varsigma}
\leq
\frac{C n^{\gamma+\varsigma} \ln n}{2^{\varsigma}\ln 2},
\quad n \geq 2,
\end{split}
\end{align}
where $\varsigma>1$ is chosen in such a way that $C\left(\zeta(\varsigma)-1\right)=1$. 
On the other hand, if $C=1$, then by Corollary \ref{cor:lem:Hodd} we obtain
$$
|f^{-1}(n)| \leq H(n) n^\gamma \leq n^{\gamma+\rho}, \quad n \geq 2,
$$
where $\rho = 1.72865\dots$ is the unique root of $\zeta(s)=2$.

Let us provide the following improvement of \eqref{eq:f-1up}.
\begin{proposition}\label{prop:f-1}
	Assume that there exist $C>0$ and $\gamma \in \mathbb{R}$ such that $|f(n)| \leq C n^\gamma$ for all $n \geq 2$. Then 
	\begin{equation}\label{eq:f-upper}
	|f^{-1}(n)| \leq n^{\gamma+\varsigma}, \quad n \geq 2,
	\end{equation}
	where $\varsigma>1$ is the unique root of $\zeta(s) = \frac{1}{C}+1$. 
\end{proposition}
\begin{proof}
	Let us prove \eqref{eq:f-upper} by induction. 
	Notice that the recurrence formula \eqref{eq:DI} can be equivalently rewritten as 
	\begin{equation}\label{eq:DI0}
	f^{-1}(1) = 1
	\quad \text{and} \quad 
	f^{-1}(n) = - \sum_{\substack{d|n \\ d>1}} f(d) f^{-1}\left(\frac{n}{d}\right), \quad n \geq 2.
	\end{equation}
	The base of induction is trivial.
	Let us fix some $n \geq 2$ and suppose that $|f^{-1}(m)| \leq m^{\gamma+\varsigma}$ for all $m < n$. 
	Then we obtain from \eqref{eq:DI0} that
	$$
	|f^{-1}(n)| 
	\leq
	\sum_{\substack{d|n \\ d>1}} C d^{\gamma} \left(\frac{n}{d}\right)^{\gamma+\varsigma} 
	=
	C \left(\sum_{d|n} \frac{1}{d^\varsigma} -1\right) n^{\gamma+\varsigma}
	\leq
	C (\zeta(\varsigma)-1) n^{\gamma+\varsigma} = n^{\gamma+\varsigma}
	$$
	since $C(\zeta(\varsigma)-1)=1$ by definition, and hence the result follows.
\end{proof}

\begin{remark}\label{rem:optimal}
The upper bounds for $|f^{-1}(n)|$ obtained in Proposition \ref{prop:f-1} and in Proposition \ref{prop:submult} above are optimal at least for $C=1$.
Let us take some $\gamma \in \mathbb{R}$ and set $f(n) = -n^\gamma$ for all $n \geq 2$. 
Then we see from \eqref{eq:Dirinv1} that $f^{-1}(n) = H(n) n^\gamma$ for all $n \geq 2$,
which is an extension of the example from \cite{hille} discussed in Section \ref{sec:intro}.
Recall that $H(n) < n^\rho$ for all $n \geq 2$, and for any $\varepsilon>0$ there exist infinitely many $n$ such that $H(n) > n^{\rho-\varepsilon}$, see Remark \ref{rem:Hn}. Thus, for any $\varepsilon>0$ there exist infinitely many $n$ such that
$$
n^{\gamma + \rho - \varepsilon} < f^{-1}(n) < n^{\gamma + \rho},
$$
which yields the optimality.
\end{remark}

\subsubsection{Exponential behaviour}\label{sec:general:exp}
Along this subsection, we will assume that $|f(n)| \leq A c^n$ for some $A, c>0$ and all $n \geq 2$. 
We easily see from \eqref{eq:Dirinv1} that, under this assumption,
\begin{equation}\label{eq:f-1c<1}
|f^{-1}(n)| \leq \sum_{k=1}^{\Omega(n)} A^k \sum_{\substack{d_1\cdots d_k=n \\ d_1,\cdots, d_k \geq 2}} |f(d_1)| \cdots |f(d_k)| 
\leq
\sum_{k=1}^{\Omega(n)} A^k c^{d^{\min}_k(n)} H_k(n)
\quad \text{for}~ c \in (0,1),
\end{equation}
and
\begin{equation}\label{eq:f-1c>1}
|f^{-1}(n)| \leq \sum_{k=1}^{\Omega(n)} A^k \sum_{\substack{d_1\cdots d_k=n \\ d_1,\cdots, d_k \geq 2}} |f(d_1)| \cdots |f(d_k)| 
\leq
\sum_{k=1}^{\Omega(n)} A^k c^{d^{\max}_k(n)} H_k(n)
\quad \text{for}~ c >1,
\end{equation} 
where
$$
d^{\min}_k(n) = \min\left\{d_1+ \cdots +d_k:~ d_1\cdots d_k = n,~ d_i \geq 2, ~d_i \in \mathbb{N}\right\}
$$
and
$$
d^{\max}_k(n) = \max\left\{d_1+ \cdots +d_k:~ d_1\cdots d_k = n,~ d_i \geq 2, ~d_i \in \mathbb{N}\right\}.
$$
Moreover, we set $d^{\min}_k(n) = \infty$ and $d^{\max}_k(n) = -\infty$ if the corresponding feasible sets are empty.

In order to estimate $|f^{-1}(n)|$ via \eqref{eq:f-1c<1} or \eqref{eq:f-1c>1}, we obtain the following lower bounds for $d^{\min}_k(n)$ and upper bound for $d^{\max}_k(n)$.
\begin{lemma}\label{lem:upperdd}
	For all $n \geq 1$ and $k \geq 1$, 
	\begin{equation}\label{eq:dminkn}
	d^{\min}_k(n) \geq k n^\frac{1}{k} \geq e \ln n
	\end{equation}
	and
	\begin{equation}\label{eq:dmaxkn}
	d^{\max}_k(n) \leq 2(k-1) + \frac{n}{2^{k-1}}.
	\end{equation}
\end{lemma}
\begin{proof}
First, let us obtain the lower bounds for $d^{\min}_k(n)$. 
If we assume that $n=2^k$, then $d^{\min}_k(n) = 2k$, and if $n < 2^k$, then $d^{\min}_k(n) = \infty$, i.e., the first inequality of \eqref{eq:dminkn} is satisfied for $n \leq 2^k$.
Therefore, let us assume that $n > 2^k$.
Notice that $d^{\min}_k(n) \geq D^{\min}_k(n)$, where
$$
D^{\min}_k(n) = \min\left\{d_1+ \cdots +d_k:~ d_1\cdots d_k = n,~ d_i \geq 2, ~d_i \in \mathbb{R}\right\}.
$$
That is, in the definition of $D^{\min}_k(n)$ we allow each $d_i$ to be non-natural.
It is not hard to see that $D^{\min}_k(n)$ has a solution $(d_1, \dots, d_k)$.
Therefore, $(d_1, \dots, d_k)$ satisfies the Karush--Kuhn--Tucker conditions, see, e.g., \cite{kkt}. Namely, there exist $\lambda_0 \in \mathbb{R}$ and $\lambda_i \leq 0$ for $i = 1, \dots, k$ such that 
\begin{equation}\label{eq:kkt1}
1 + \frac{\lambda_0 n}{d_i} + \lambda_i = 0, \quad i=1,\dots,k,
\end{equation}
and if some $d_i>2$, then $\lambda_i = 0$. 

Since each $d_i \geq 2$ and we assume that $n > 2^k$, there exists at least one $d_m > 2$. 
Thus,
\begin{equation}\label{eq:lm=0}
\lambda_m = 0
\quad \text{and} \quad
\lambda_0 = -\frac{d_m}{n}.
\end{equation}
If there exists another $d_l > 2$, $l \neq m$, then, as above,
$$
\lambda_l = 0
\quad \text{and} \quad
\lambda_0 = -\frac{d_l}{n},
$$
which yields $d_l = d_m$.
Suppose now that there exists some $d_\kappa = 2$. Taking into account \eqref{eq:lm=0} and recalling that $d_m>2$, we see from \eqref{eq:kkt1} that
$$
\lambda_\kappa = -1 + \frac{d_m}{2} > 0,
$$
which is impossible since $\lambda_i \leq 0$ for all $i=1,\dots,k$. Therefore, we conclude that each $d_i>2$ and 
$d_i=d_j$ for $i \neq j$. 
Consequently, $n=d^k$ and $d^{\min}_k(n) \geq k n^\frac{1}{k}$. 

Let us now estimate $k n^\frac{1}{k}$ from below by $e \ln n$.
To this end, we fix some $n \geq 2$ and consider the function
$$
G(x) = x n^\frac{1}{x}, \quad x \in \mathbb{R},~ x>0.
$$
It is clear that $G(x)$ has exactly one global minimizer $x_0 = \ln n$ with $G(x_0) = e \ln n$.
Hence, we conclude that
$$
d^{\min}_k(n) \geq k n^\frac{1}{k} \geq e \ln n, \quad n \in \mathbb{N}.
$$

Second, let us obtain the upper bound for $d^{\max}_k(n)$ by following the same strategy as above.
Evidently, for $n \leq 2^k$ the inequality \eqref{eq:dmaxkn} holds true and we can restrict ourselves to $n > 2^k$.
Denoting
$$
D^{\max}_k(n) = \max\left\{d_1+ \cdots +d_k:~ d_1\cdots d_k = n,~ d_i \geq 2, ~d_i \in \mathbb{R}\right\},
$$
we see that $d^{\max}_k(n) \leq D^{\max}_k(n)$.
Every solution $(d_1, \dots, d_k)$ to $D^{\max}_k(n)$ fulfils the Karush--Kuhn--Tucker conditions, i.e., there are $\lambda_0 \in \mathbb{R}$ and $\lambda_i \leq 0$ for $i = 1, \dots, k$ satisfying
\begin{equation*}\label{eq:kkt2}
-1 + \frac{\lambda_0 n}{d_i} + \lambda_i = 0, \quad i=1,\dots,k,
\end{equation*}
and if some $d_i>2$, then $\lambda_i = 0$. 

Recalling that $n > 2^k$, we can find $d_m>2$. Arguing as above, we deduce that if there is another $d_l>2$ with $l \neq m$, then $d_l=d_m$. 
Therefore, we can assume, without loss of generality, that $d_1=\dots=d_{r}>2$ for some $r \in \{1,\dots,k\}$, and $d_{r+1}=\dots=d_k=2$ if $r<k$. This implies that $n=d_1^{r} 2^{k-r}$, and hence $d_1=\left(\frac{n}{2^{k-r}}\right)^\frac{1}{r}$ and
\begin{equation}\label{eq:Dmax}
D^{\max}_k(n) = 2(k-r) + \frac{r n^\frac{1}{r}}{2^\frac{k-r}{r}}.
\end{equation}
To determine the actual value of $r \in \{1,\dots,k\}$, let us maximize the right-hand side of \eqref{eq:Dmax} with respect to $r$. 
To this end, consider the function
$$
G(x) = 2(k-x) + \frac{x n^\frac{1}{x}}{2^\frac{k-x}{x}}, \quad x \in \mathbb{R}, ~ x>0.
$$
We see that
$$
G'(x) = \frac{n^\frac{1}{x}}{2^\frac{k-x}{x}}
\left(
1-\left(\frac{2^k}{n}\right)^\frac{1}{x} + \ln\left(\left(\frac{2^k}{n}\right)^\frac{1}{x}\right)
\right)\leq 0, \quad x>0,
$$
and $G'(x)=0$ if and only if $n = 2^k$. 
Since $n > 2^k$, $G(x)$ decreases with respect to $x>0$, and we conclude that $r=1$ is the unique maximizer for the right-hand side of \eqref{eq:Dmax}, which yields the claimed upper bound \eqref{eq:dmaxkn}.
\end{proof}

Combining now \eqref{eq:f-1c<1} or \eqref{eq:f-1c>1} with Lemma \ref{lem:upperdd}, we obtain the following result.
\begin{proposition}
	Assume that there exist $A, c>0$ such that $|f(n)| \leq Ac^n$ for all $n \geq 2$.
	If $c \in (0,1)$, then
	\begin{equation}\label{eq:f-1:pol3}
	|f^{-1}(n)|
	\leq
	\frac{\Omega(n)\cdot An^{\varsigma + e \ln c}}{2^{\varsigma}}
	\leq
	\frac{An^{\varsigma + e \ln c}\ln n}{2^{\varsigma}\ln 2},
	\quad n \geq 2,
	\end{equation}
	where $\varsigma>1$ is the unique root of $\zeta(s)=\frac{1}{A}+1$.
	If $c \in (0,1)$ but $A \leq 1$, then  
	\begin{equation}\label{eq:f-1:pol}
	|f^{-1}(n)| \leq n^{\rho + e \ln c}, \quad n \geq 2,
	\end{equation}
	where $\rho = 1.72865\dots$ is the unique root of $\zeta(\rho) = 2$.
	
	If $c>1$, then there exists $\widetilde{A}>0$ such that
	\begin{equation}\label{eq:f-1:pol2}
	|f^{-1}(n)|
	\leq
	Ac^n + \frac{(\Omega(n)-1)An^{\upsilon}}{2^{\upsilon}}c^{\frac{n}{2}}
	\leq
	\widetilde{A} c^n, \quad n \geq 2,
	\end{equation}
	where $\upsilon>1$ is the unique root of $\zeta(s)=\frac{1}{Ac^2}+1$.
\end{proposition}
\begin{proof}
	First, assume that $c \in (0,1)$. We estimate \eqref{eq:f-1c<1} by \eqref{eq:upperH} and \eqref{eq:dminkn} as
	$$
	|f^{-1}(n)|
	\leq 
	\sum_{k=1}^{\Omega(n)} A^k c^{d^{\min}_k(n)} H_k(n)
	\leq
	\frac{A n^s c^{e \ln n}}{2^{s}} \sum_{k=1}^{\Omega(n)} (A(\zeta(s)-1))^{k-1}, \quad n \geq 2,
	$$
	for every $s>1$.
	Taking $s=\varsigma$, we have $A(\zeta(\varsigma)-1)=1$ and \eqref{eq:f-1:pol3} follows directly.
	Under the additional assumption $A \leq 1$, we easily get \eqref{eq:f-1:pol} from \eqref{eq:f-1c<1} and Corollary \ref{cor:lem:Hodd}:
	$$
	|f^{-1}(n)| \leq 
	\sum_{k=1}^{\Omega(n)} A^kc^{d^{\min}_k(n)} H_k(n) \leq c^{e \ln n} \sum_{k=1}^{\Omega(n)} H_k(n) = n^{e \ln c} H(n) \leq n^{e \ln c + \rho}, \quad n \geq 2.
	$$
	
	Second, assume that $c>1$. With $d_1^{\max}(n)=n$, $H_1(n)=1$, \eqref{eq:upperH}, and \eqref{eq:dmaxkn}, we deduce
	\begin{align}
	\notag
	|f^{-1}(n)|
	&\leq
	Ac^n + \sum_{k=2}^{\Omega(n)}A^kc^{2(k-1)+\frac{n}{2^{k-1}}}\frac{(\zeta(s)-1)^{k-1}n^s}{2^s}\\
	\label{eq:c_asymp}
	&\leq
	Ac^n + \frac{An^sc^\frac{n}{2}}{2^s} \sum_{k=2}^{\Omega(n)}(Ac^2(\zeta(s)-1))^{k-1},
	\quad n \geq 2,
	\end{align}
	for every $s>1$. For $s=\upsilon$ with $Ac^2(\zeta(\upsilon)-1)=1$ we obtain the first inequality in \eqref{eq:f-1:pol2}. 
	Obviously, the first term in \eqref{eq:c_asymp} is the leading term as $n \to \infty$, and hence the second inequality in \eqref{eq:f-1:pol2} is valid, as well.  
\end{proof}

\begin{remark}
	Comparing the upper bound \eqref{eq:f-1:pol} for a general $f(n)$ with the upper bound \eqref{eq:expf1} for a multiplicative $f(n)$, we see that \eqref{eq:expf1} provides the better asymptotic. 
	On the other hand, \eqref{eq:f-1:pol} provides an improvement of the following upper bound obtained in the proof of \cite[Theorem 3]{sowa}: if $|f(1)| \geq \frac{c}{2}$ and $|f(n)| \leq c^n$ for $n \geq 2$ where $c \leq \left(2 \cdot 3^\frac{4}{3} + 3^\frac{2}{3}\right)^{-1}$, then
	$$
	|f^{-1}(n)| \leq \frac{2}{c\, n^2}, 
	\quad n \geq 2.
	$$
\end{remark}

\subsection{Miscellaneous cases}\label{sec:misc}
In this section, we consider the asymptotic	behaviour of $f^{-1}(n)$ for several special classes of $f(n)$. We limit ourselves to the consideration of a polynomial bound for $|f(n)|$.

\subsubsection{Truncated $f(n)$}
First, let us assume that there exists $N \geq 2$ such that $f(n) = 0$ for all $2 \leq n \leq N$. 
We see from \eqref{eq:Dirinv1} that $f^{-1}(n) = 0$ for all $2 \leq n \leq N$, and $f^{-1}(n)$ depends only on the values of $f(d)$ with $d > N$. 
Therefore, we can argue in much the same way as in Proposition \ref{prop:f-1} to obtain the following result.
\begin{proposition}\label{prop:f-5}
	Assume that there exist $N \geq 2$, $C>0$, and $\gamma \in \mathbb{R}$ such that $f(n)=0$ for all $2 \leq n \leq N$ and $|f(n)| \leq Cn^\gamma$ for all $n > N$. 
	Then 
	$$
	|f^{-1}(n)| \leq n^{\gamma+\varsigma},
	\quad n > N,
	$$
	where $\varsigma>1$ is the unique root of $\zeta(s) = \frac{1}{C}+ \sum_{m=1}^{N}\frac{1}{m^s}$.
\end{proposition}

Second, let us assume that there exists $N \geq 2$ such that $f(n) = 0$ for all $n \geq N+1$. 
Clearly, in this case $f^{-1}(n)$ depends only on the values of $f(d)$ with $d < N$. 
Arguing along the same lines as in Proposition \ref{prop:f-1}, we derive the following result.
\begin{proposition}\label{prop:f-7}
	Assume that there exist $N \geq 2$, $C>0$, and $\gamma \in \mathbb{R}$ such that $f(n)=0$ for all $n \geq N+1$ and $|f(n)| \leq Cn^\gamma$ for all $n \leq N$. 
	Then 
	$$
	|f^{-1}(n)| \leq n^{\gamma+\varsigma},
	\quad n \geq 2,
	$$
	where $\varsigma > 0$ is the unique root of $\sum_{m=2}^{N}\frac{1}{m^s} = \frac{1}{C}$. 
\end{proposition}

\subsubsection{$f(n)$ supported on odd $n$}\label{sec:fodd}
Assume that $f(n)$ represents the Fourier coefficients of a function $F \in L^2(0,1)$, that is, $F(x) = \sum_{n=1}^{\infty} f(n) \sin(n \pi x)$. 
If one is interested in the basisness or completeness of the system $\{F(nx)\}$ (see Section \ref{sec:intro}), then it seems natural to choose $F$ such that it is symmetric with respect to the point $x=1/2$, see, e.g., generalized trigonometric functions \cite{bind}.
Under this symmetry assumption, we have $f(n)=0$ for all even $n$ and it is clear from \eqref{eq:Dirinv1} that $f^{-1}(n)=0$ for all even $n$, as well.
Arguing as in the proof of Proposition \ref{prop:f-1} and using \eqref{eq:zetaodd}, we get the following result.
\begin{proposition}\label{prop:f-2}
	Assume that $f(n)=0$ for all even $n$ and there exist $C>0$ and $\gamma \in \mathbb{R}$ such that $|f(n)| \leq Cn^\gamma$ for all $n \geq 2$. Then 
	$$
	|f^{-1}(n)| \leq n^{\gamma+\varsigma}, \quad n \geq 2,
	$$
	where $\varsigma>1$ is the unique root of $\left(1-\frac{1}{2^s}\right)\zeta(s) = \frac{1}{C}+1$. 
	In particular, if $C=1$, then $\varsigma = \eta  = 1.37779\dots$ 
\end{proposition}

\end{document}